\documentclass[journal]{IEEEtran}
\usepackage{anysize}
\marginsize{1.9cm }{1.9cm }{1.1cm }{2.8cm }
\usepackage{amsfonts,amssymb}
\usepackage[cmex10]{amsmath}
\usepackage{epsfig}
\usepackage{amsfonts}
\usepackage{graphics}
\usepackage{graphicx}
\usepackage{subfig}
\usepackage{mathtools}
\usepackage{enumerate}
\usepackage{cite}
\epsfverbosetrue

    \def\independenT#1#2{\mathrel{\setbox0\hbox{$#1#2$}%
    \copy0\kern-\wd0\mkern4mu\box0}}

\title{\vspace{.64cm}Delay Optimal Server Assignment to Symmetric Parallel Queues with Random Connectivities}  % Declares the document's title.
\author{Hassan Halabian, Ioannis Lambadaris, Chung-Horng Lung
\\
Department of Systems and Computer Engineering\\
Carleton University, 1125 Colonel By Drive, Ottawa, ON, K1S 5B6, Canada\\
Email: \{hassanh, ioannis, chlung\}@sce.carleton.ca 
}      % Declares the author's name.
\begin{document}             % End of preamble and beginning of text.
\maketitle  
\thispagestyle{empty}                % Produces the title.
%%%%%%%%%%%%%%%%%%%%%%%%%%%%%%%%%%%%%%%%%%%%%%%%%%%%%%%%%%%%%%%%%
\newtheorem{theor}{Theorem}
\newtheorem{lem}{Lemma}
\newtheorem{pro}{Proposition}
\newtheorem{defin}{Definition}
\newtheorem{conj}{Conjecture}
\newtheorem{cor}{Corollary}

%%%%%%%%%%%%%%%%%%%%%%%%%%%%%%%%%%%%%%%%%%%%%%%%%%%%%%%%%%%%%%%%%
\IEEEpeerreviewmaketitle
\begin{abstract}
In this paper, we investigate the problem of assignment of $K$ identical servers to a set of $N$ parallel queues in a time slotted queueing system. The connectivity of each queue to each server is randomly changing with time; each server can serve at most one queue and each queue can be served by at most one server per time slot. Such queueing systems were widely applied in modeling the scheduling (or resource allocation) problem in wireless networks.
It has been previously proven that Maximum Weighted Matching (MWM) is a throughput optimal server assignment policy for such queueing systems \cite{Tassiulas92,sarkar}. In this paper, we prove that for a symmetric system with i.i.d. Bernoulli packet arrivals and connectivities, MWM minimizes, in \textit{stochastic ordering} sense, a broad range of cost functions of the queue lengths including total queue occupancy (or equivalently average queueing delay).

\end{abstract}
\section{Introduction}
Optimal stochastic control of emerging wireless networks is one of the primary objectives in the design of such networks. In general, the main goal in the stochastic control of wireless networks is to distribute the shared resources in physical (e.g. power) and MAC layers (e.g. radio interfaces, relay stations and orthogonal channels) to multiple users such that a certain stochastic performance attribute is optimized.
While various performance attributes including the stable throughput region, power consumption and utility functions of the admitted rates have been studied in many papers, average queueing delay has been considered far less in literature.
This is due to the inherent difficulty of delay optimal scheduling problems in queueing systems with time varying channel conditions. 
In this paper, we consider a discrete time queueing system which is suitable in modeling of orthogonal resource assignment (e.g. radio interfaces/channel allocation) in multi-user wireless access networks. In our system, we model the available shared resources by a set of identical servers. The model also consists of a set of queues whose connectivities to each server is changing by time randomly. Therefore, the resource assignment problem is equivalent to finding a \textit{matching} between the queues and the servers at each time slot such that some performance objectives are optimized. 
%The main goal of this paper is to determine a matching at each time slot such that a broad range of cost functions of the queue lengths is minimized.
It has been already shown that Maximum Weighed Matching (MWM) is throughput optimal for such a system, i.e., it maximizes the stable throughput region of the system \cite{Tassiulas92,sarkar}. MWM has also been extensively used in literature for treating the scheduling problem in crossbar packet switches \cite{mckeown,Tassiulas98,neely2004,Leonardi}.
In this paper, we prove that for a symmetric system with i.i.d. Bernoulli arrivals and connectivities (i.e. with the same arrival and connectivity parameters for all the queues), MWM is also optimal in minimizing, in \textit{stochastic ordering} sense, a broad range of cost functions of queue lengths including total queue occupancy (or equivalently average queueing delay)\footnote{We order two discrete time random processes $A=\{A(t)\}_{t=1}^{\infty}$ and $B=\{B(t)\}_{t=1}^{\infty}$ stochastically as follows: We say $A$ is stochastically less than $B$ and we write $A \leq_{st} B$ if $\Pr (A(t)>r) \leq \Pr(B(t)>r)$ for all $t=1,2,...$ and all $r \in \mathbb{R}$. The notion and relevant properties will be discussed in more detail in Section \ref{sto}.}. In other words, we show that MWM policy minimizes stochastically a broad range of cost functions of queue length processes including the expected total queue occupancy across all possible server assignment policies.  

The problem of optimal server allocation in queueing systems with random connectivities was mainly addressed in \cite{Tassiulas92,sarkar,Tassiulas93,ganti,javidi,javidi4,javidi2, hussein, khodam}. 
In \cite{Tassiulas92}, the authors introduced the notion of stability region of a general queueing network with 
time varying connectivities and they proposed back-pressure algorithm as a throughput optimal resource allocation policy for queueing networks. In \cite{Tassiulas93}, they considered a multi-queue single-server queueing system with random connectivities. They characterized the stability region by a set of linear inequalities and also proved that for a symmetric system with the same arrival and connectivity parameters for all the queues, LCQ (Longest Connected Queue) provides the optimal performance in terms of average queue occupancy.

In \cite{javidi2}, Maximum Weight (MW) policy was proposed as a throughput optimal server allocation policy for multi-queue multi-server queueing systems with stationary channel processes. In \cite{khodam}, the authors characterized the network capacity region of multi-queue multi-server queueing systems with time varying connectivities. They also obtained an upper bound for the average queueing delay of AS/LCQ policy which is a throughput optimal server allocation policy for these systems. The results were further extended in \cite{ISIT-2011} for more general stationary channel distributions (and not just i.i.d. Bernoulli channels).

The authors in \cite{ganti} considered a queueing model with a set of symmetrical parallel queues competing for $K$ identical servers. The connectivity of each queue to all the servers is assumed to be the same at each time slot and during each time slot, each queue can attract at most one server. The authors proposed LCQ policy in which the servers are allocated to the $K$ longest connected queues at each time slot. They proved the optimality of LCQ policy by using dynamic coupling and stochastic ordering method.

The work in \cite{javidi,javidi3,javidi4,hussein} focuses on the optimal server allocation problem in multi-queue multi-server queueing systems in terms of average queueing delay. In \cite{javidi,javidi3,javidi4}, the authors
% argue that in general, achieving instantaneous throughput and load balancing is impossible in a general MQMS system. However, they showed that this goal is attainable in the special case with ON-OFF channel processes. They also 
introduced MTLB (Maximum-Throughput Load-Balancing) policy and showed that this policy minimizes a class of cost functions including total average delay for the case of two symmetric queues. The work in \cite{hussein} considers this problem for general number of symmetric queues and servers. In \cite{hussein}, a class of \textit{Most Balancing} (MB) policies was characterized among all work conserving policies which are minimizing, in stochastic ordering sense, a class of cost functions including total average delay. 
Note that in the model used in \cite{javidi,javidi3,javidi4,hussein,khodam}, there is no restriction on the number of servers that are serving a queue at each time slot.
In \cite{sarkar}, it was shown that for a multi-queue multi-server system in which queues are restricted to attract at most one server at each time slot, Maximum Weighted Matching (MWM) policy is throughput optimal. The authors also considered the effect of infrequent channel state measurements on the stability region. 
%In \cite{mckeown}, policy was also 

The rest of the paper is organized as follows. Section \ref{modeldes} describes the model and the notation required throughout the paper. In section \ref{background}, we introduce Maximum Weighted Matching (MWM) policy as the optimal policy for the described model. We will also review the concepts of stochastic ordering and dynamic coupling method which are the main mathematical tools used in proving the optimality of MWM policy. In section \ref{asli}, we present the main result of this paper, that is proving the optimality of MWM server assignment policy. Section \ref{conc} summarizes the conclusions of the paper.

%%%%%%%%%%%%%%%%%%%%%%%%%%%%%%%%%%%%%%%%%%%%%%%%%%%%%%%%%%%%%%%%%%%%%%%%%%%%%%%%%%%%%%%%%%%%%
%%%%%%%%%%%%%%%%%%%%%%%%%%%%%%%%%%%%%%%%%%%%%%%%%%%%%%%%%%%%%%%%%%%%%%%%%%%%%%%%%%%%%%%%%%%%%
%\section{Notations}
%Before we proceed to next section, we introduce basic notations often used through the paper. Other notations are introduced whenever it is necessary. Through this paper, all the vectors are considered to be row vectors. The time average of a function $f(t)$ is denoted by $ \overline{f(t)} $, i.e.,  $\overline{f(t)}= \lim_{t\rightarrow \infty}\frac{1}{t} \sum_{\tau=1}^{t} f(\tau)$. The operator $``\circledast"$ is used for entry-wise multiplication of two matrices. By $\underline{1}_K$, we denote a row vector of size $K$ whose elements are all identically equal to $1$. The Expectation of random processes (or random variables) is denoted by $E[\cdot]$. The cardinality of a set is shown by $|\cdot|$. The operator for inner product of two vectors is $\left\langle \cdot, \cdot \right\rangle$. The boundary of a set is represented by $bound (\cdot)$. For any vector $\alpha=(\alpha_1,\alpha_2,...,\alpha_N)$ and a non-empty index set $ \mathcal{U}=\{u_1,u_2,...,u_{|\mathcal{U}|}\} \subseteq \{1,2,...,N\} $ and $u_1<u_2<...< u_{|\mathcal{U}|}$, we define $ \alpha_{\mathcal{U}}=(\alpha_{u_1},\alpha_{u_2},...,\alpha_{u_{|\mathcal{U}|}}) $.

%%%%%%%%%%%%%%%%%%%%%%%%%%%%%%%%%%%%%%%%%%%%%%%%%%%%%%%%%%%%%%%%%%%%%%%%%%%%%%%%%%%%%%%%%%%%%%
%%%%%%%%%%%%%%%%%%%%%%%%%%%%%%%%%%%%%%%%%%%%%%%%%%%%%%%%%%%%%%%%%%%%%%%%%%%%%%%%%%%%%%%%%%%%%%

\section{Model Description}
\label{modeldes}
%
%\section {Model Description} \label{modeldes}
%\begin{figure}[tp]
%    \centering
%    \includegraphics[width=0.45\textwidth]{model}
%	\caption{ Multi-queue multi-server queueing system with stationary channel distribution}
%	\label{model}
%\end{figure} 
We consider a time slotted parallel queueing system with a set of parallel symmetrical
queues $\mathcal{N}=\{1,2,...,N\}$ and infinite buffer space for each queue. 
Packets in this system are assumed to have constant length and require one time slot to complete service. The service to this set of queues is provided through a set of identical servers namely $\mathcal{K}=\{1,2,...,K\}$. 
%(see Figure \ref{model}). 
The connectivity of each queue $n\in \cal N$ to each server $k \in \cal K$ at each time slot $t$ is random and follows a Bernoulli distribution. We denote the connectivity of queue $n$ to server $k$ at time slot $t$ by $C_{n,k}(t)$. Note that $C_{n,k}(t) \in \{0,1\}$ and $E[C_{n,k}(t)]=p $ for all $n \in \cal N$ and $k \in \cal K$ and $t=1,2,...$.
% We arrange all the connectivity random variables at each time $t$ in an $N\times K$ matrix as $C(t)=(C_{n,k}(t)),\forall n \in {\cal N}, \forall k \in \cal K$.

At each time slot, each server can serve at most one packet from a connected non-empty queue.
Note that in the system we do not have server sharing i.e., a server can serve at most one queue at each time slot. We also assume that a queue which is being serviced by a server at a given time slot, cannot get service from other servers during the same time slot.
\begin{figure}[tp]
    \centering
    \includegraphics[width=.4\textwidth]{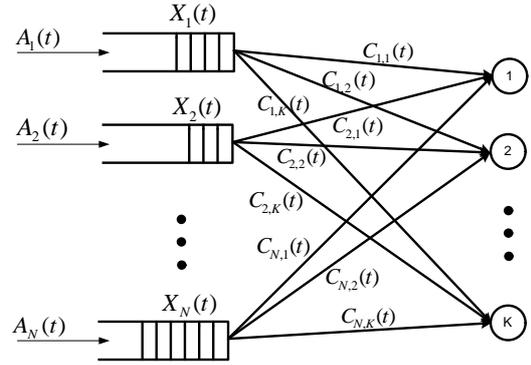}
    \caption{Discrete time queueing system with $N$ parallel queues and $K$ servers}
    \label{model}
\end{figure} 
%
%The amount of service offered by a given service vector $(s_{1,k_1},s_{2,k_2},...,s_{M,k_M})$ to each queue $n$ at time slot $t$ is determined by function $C_n(s_{1,k_1},s_{2,k_2},...,s_{M,k_M},t)$. In our model $C_n(s_{1,k_1},s_{2,k_2},...,s_{M,k_M},t)$ is modelled as a Bernoulli random process with parameter $p$ for all $n,k_1,k_2,...,k_M$,
%  % $1 \leq n \leq N,1 \leq k_1 \leq K_1, 1 \leq k_2 \leq K_2,...,1 \leq k_M \leq K_M$
%i.e., for any service vector $(s_{1,k_1},s_{2,k_2},...,s_{M,k_M})$ and any queue $n$, we have $C_n(s_{1,k_1},s_{2,k_2},...,s_{M,k_M},t) \in \{0,1\}~\forall t$ and $E[C_n(s_{1,k_1},s_{2,k_2},...,s_{M,k_M},t)]=p$. Henceforth, we call $C_n(s_{1,k_1},s_{2,k_2},...,s_{M,k_M},t)$ as the connectivity variable of user $n$ if it uses service vector $(s_{1,k_1},s_{2,k_2},...,s_{M,k_M})$. At each time slot $t$, the network controller decides about which service vector $(s_{1,k_1},s_{2,k_2},...,s_{M,k_M})$ should be allocated to each queue $n$. Note that in our system sever sharing is not allowed. In other words, if the service vector $(s_{1,k_1},s_{2,k_2},...,s_{M,k_M})$ and $(s_{1,k'_1},s_{2,k'_2},...,s_{M,k'_M})$ are allocated to queues $n$ and $n'$ (for any $n \neq n'$) respectively, then $k_m \neq k'_m$ for all $1 \leq m\leq M$. We call the service vectors with this property as \textit{non-overlapping} service vectors. We also assume that at most one series of servers $q \in \cal S$ can be assigned to each queue at each time slot.

Let $A_n(t)$ be the packet arrival process (number of packet arrivals) to queue $n$ at time slot $t$. We assume that new arrivals at each time slot are added to the queues at the end of the time slot. Assume that the arrival processes $A_n(t)$ at each time slot $t$ are independent Bernoulli random variables with the same parameter for all $n$ and $t$. We denote the length of queue $n$ at the end of time slot $t$ (i.e., after adding the new arrivals) by $X_n(t)$. In other words, $X_n(t)$ represents the number of packets in the $n$th queue at the end of time slot $t$ (or beginning of time slot $t+1$).

A server assignment policy at each time slot determines an assignment of servers of set $\cal K$ to the queues of set $\cal N$. In other words, at each time slot the scheduler has to decide about a \textit{bipartite matching} (matching in bipartite graphs) between sets $\cal N$ and $\cal K$.
This should be accomplished based on the available information about the connectivities $C_{n,k}(t)$ and also the queue length process at the beginning of time slot $t$ (which is $X(t-1)=(X_1(t-1),X_2(t-1),...,X_N(t-1))$).
%The assignment (or matching) of servers to the queues must be such that some stochastic performance attribute (e.g., throughput, delay, etc.) is optimized.
%Thus, a server assignment policy $\pi$ is defined as a rule that determines a feasible matching at each time slots $t$.
For a given policy $\pi$, suppose that indicator variable $I^{(\pi)}_{n,k}(t)$ is defined to be $``1"$ if server $k$ is assigned to queue $n$ at time slot $t$ and $``0"$ otherwise. We define $M^{(\pi)}(t)=\{I^{(\pi)}_{n,k}(t), \forall n \in \mathcal{N},k\in \mathcal{K}\}$ as the employed \textit{matching} by policy $\pi$ at time slot $t$.
Therefore, a server scheduling policy $\pi$ is defined as $\pi=\{M^{(\pi)}(t)\}_{t=1}^{\infty}$.

%
%\begin{eqnarray}
%\label{conditions}
%&\displaystyle\sum_{k_1=1}^{K_1} \sum_{k_2=1}^{K_2} \sum_{k_3=1}^{K_3} \cdots \sum_{k_M=1}^{K_M} I_n(k_1,k_2,...,k_M,t)\leq 1 ~~ (n=1,2..,N) \nonumber\\
%&\displaystyle\sum_{n=1}^{N} \sum_{k_2=1}^{K_2} \sum_{k_3=1}^{K_3} \cdots \sum_{k_M=1}^{K_M} I_n(k_1,k_2,...,k_M,t)\leq 1 ~~ (k_1=1,2..,K_1) \nonumber \\
%&\displaystyle\sum_{n=1}^{N} \sum_{k_1=1}^{K_1} \sum_{k_3=1}^{K_3} \cdots \sum_{k_M=1}^{K_M} I_n(k_1,k_2,...,k_M,t)\leq 1 ~~ (k_2=1,2..,K_2) \nonumber 
%\\ &\vdots \nonumber & \\
%&\displaystyle\sum_{n=1}^{N} \sum_{k_1=1}^{K_1} \sum_{k_2=1}^{K_2} \cdots \sum_{k_{M-1}=1}^{K_{M-1}} I_n(k_1,k_2,...,k_M,t)\leq 1 ~~ (k_M=1,2..,K_M)
%\end{eqnarray}
% 
%As we can see from (\ref{conditions}), the scheduling problem in the introduced model is equivalent to finding an appropriate Multi-index assignment in the graph of Figure \ref{graph}. In the special case where $M=1$, the problem is to determine a matching in a bipartite graph.
According to the above discussion, we can see that the queue length random variable $X_n(t)$, $\forall n \in \cal N$ evolves with time according to the following rule:
\begin{eqnarray} \label{evol}
X_n(t)= \left (X_n(t-1)-\displaystyle\sum_{k=1}^{K}C_{n,k}(t)I_{n,k}^{(\pi)}(t) \right)^+ + A_n(t) \nonumber
\end{eqnarray}
where $(\cdot)^+$ returns the term inside the brackets if it is non-negative and zero otherwise. Note that a server can be assigned to an empty queue however it cannot serve it since there is no packet to be served. That is why we have used operator $(\cdot)^+$ in (\ref{evol}).

As we discussed earlier, the queueing model introduced in this section is useful in modeling the resource assignment problem in various systems with shared resources. In wireless communication systems, communication resources such as communication sub-channels, relay stations, etc. are shared among users and therefore can be studied using our model (e.g. \cite{hassan-milcom2010,sarkar}). 
%Note that the time varying nature of wireless media is also captured in our model by the connectivity variables.
Bipartite Matching also has been extensively used in literature (e.g. \cite{mckeown,Tassiulas98,neely2004,Leonardi}) to model the scheduling problem in crossbar packet switching systems.
In this paper, random variables are represented by CAPITAL letters and lower case letters are used to represent sample values of the random variables.  
\section{Background} \label{background}

\subsection{Maximum Weighted Matching}     
% Multi-index assignment problems are studied as part of graph theory and was first introduced in \cite{}. In the special case were the $M=1$, an assignment is equivalent a matching in bipartite graphs \cite{}. In analogy to the matching problem in bipartite graphs, a $M+1$-index assignment problem is sometimes called a matching in $M+1$-partitive graphs.
In \cite{Tassiulas92,phd,power,Now,sarkar}, it was shown that Back-pressure algorithm maximizes the stable throughput region of a general data network. For the model introduced in section \ref{modeldes}, Back-pressure algorithm is equivalent to solving the following optimization problem at each time slot $t$ \cite{sarkar}.
\begin{eqnarray}\label{matching}
\texttt{Maximize}&~~~ \displaystyle\sum_{n=1}^{N} x_n(t-1)\displaystyle\sum_{k=1}^{K} I_{n,k}(t)c_{n,k}(t)~~~~~~~ \nonumber \\
\texttt{s.t.}&\displaystyle\sum_{k=1}^{K} I_{n,k}(t)\leq 1 ~~ (n=1,2..,N) \nonumber\\
&\displaystyle\sum_{n=1}^{N} I_{n,k}(t)\leq 1 ~~ (k=1,2..,K) 
\end{eqnarray} 
where $x_n(t-1)$ and $c_{n,k}(t)$ are the values of random variables $X_n(t-1)$ and $C_{n,k}(t)$ at time slots $t-1$ and $t$, respectively.
Note that finding the solutions of problem (\ref{matching}) is equivalent to finding a maximum weighted matching in the bipartite graph $G_t=(\mathcal{N}, \mathcal{K},\mathcal{E})$ (see Figure \ref{matching-graph}). In $G_t$, $\mathcal{N}$ and $\mathcal{K}$ are the two sets of vertices in each part of the graph and $\mathcal{E}=\{e_{n,k},\forall n \in \mathcal{N}, \forall k \in \mathcal{K}\}$ is the set of edges between these two parts. Note that the associated weight to each edge $e_{n,k}$ is $ x_n(t-1)c_{n,k}(t) $. A matching in graph $G_t$ is basically a sub-graph of $G_t$ in which no two edges share a common vertex. Note that any matching $M^{(\pi)}(t)$ at any time slot $t$ is corresponding to a sub-graph of $G_t$ namely $G_t^{(\pi)}=(\mathcal{N},\mathcal{K}, \mathcal{E}^{(\pi)})$ in which  $e_{n,k} \in \mathcal{E}^{(\pi)}$ if and only if $I^{(\pi)}_{n,k}(t)=1$.
%In general solving the $M+1$-index assignment problem of (\ref{assig}) is NP-hard. However for $M=1$, the assignment problem is equivalent to maximum weighted bipartite matching problem in graph of Figure (\ref{bi-graph}).
%Finding the Maximum matching in bipartite graphs is not NP-hard and has polynomial time solutions \cite{}. 
%Henceforth, we use the term ``Maximum Weighted Matching" (MWM) as the solution of problem (\ref{assig}) for all $M \geq 1$ (and not just for $M=1$).
Suppose that $M^{(\text{MWM})}(t)=\{I^{(\text{MWM})}_{n,k}(t), \forall n \in \mathcal{N},k\in \mathcal{K}\}$ be the matching whose indicator variables are the solution of the optimization problem (\ref{matching}). Thus, we define Maximum Weighted Matching (MWM) server assignment policy as $\text{MWM}=\{M^{(\text{MWM})}(t)\}_{t=1}^{\infty}$. 
\begin{figure}[tp]
    \centering
    \includegraphics[width=.35\textwidth]{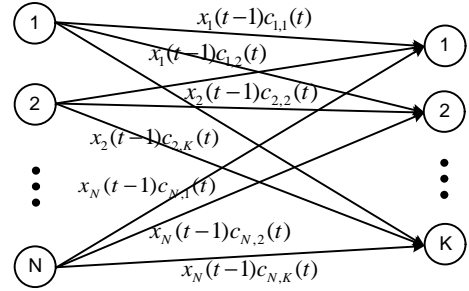}
    \caption{Bipartite graph corresponding to problem (\ref{matching})} 
    \label{matching-graph}
\end{figure}
There are several algorithms to find the maximum weighted matching in bipartite graphs. The most well known algorithm is Hungarian algorithm whose complexity is of $O((\min\{N,K\})(\max\{N,K\})^2)$ \cite{hungarian}.

As explained before, MWM is known to be throughput optimal for the queueing system described in section \ref{modeldes} \cite{sarkar}. Our contribution in this paper is to prove that MWM is also optimal in minimizing, in stochastic ordering sense, a class of cost functions of the queue length processes including the total system occupancy (or equivalently total average queueing delay) for the symmetric queueing system of Figure \ref{model} (which can be used to model a homogeneous wireless access network). We will introduce a detailed description of those class of cost functions in the following section. 

\subsection{Stochastic Ordering and Dynamic Coupling} \label{sto}
In this section, we briefly review the concepts of stochastic ordering (stochastic dominance) and dynamic coupling techniques. Consider two discrete time stochastic processes $A=\{A(t)\}_{t=1}^{\infty}$ and $B=\{B(t)\}_{t=1}^{\infty}$ in $\mathbb{R}$. We say $A$ is stochastically less than $B$ and we write $A \leq_{st} B$ if $\Pr (A(t)>r) \leq \Pr(B(t)>r)$ for all $t=1,2,...$ and all $r \in \mathbb{R}$ \cite{stoyan,ross}. Some properties of stochastic ordering are the following.
If $A\leq_{st} B$ then $f(A)\leq_{st} f(B)$ for all non-decreasing functions $f$. If $A\leq_{st} B$ then $E[A(t)] \leq E[B(t)]$.
$A$ is stochastically smaller than $B$ ($A \leq_{st} B$), if there exists process $\tilde{A}=\{\tilde{A}(t)\}_{t=1}^{\infty}$ defined on the same probability space as $B$ with the same probability distribution as $A$ and satisfy $\tilde{A}(t)\leq B(t)$ almost surely for every $t=1,2,...$ \cite{ganti}. The last statement is known as coupling of $A$ and $\tilde{A}$. In fact, when applying coupling technique, we are given the process $A$ and we try to construct a coupled process $\tilde{A}$ with the same distribution as $A$ and $\tilde{A}(t)\leq B(t)$ a.s. for all $t$. This gives us a tool for comparing processes $A$ and $B$ stochastically. This is specially useful when it is infeasible to derive the distributions of $A$ and $B$ (e.g. in our queueing model when comparing the total occupancy process for different server assignment policies).  

\section{Optimality of MWM}   \label{asli}  
In this section, we present the main result of this paper, that is proving the optimality of MWM with respect to minimization of a class of cost functions of queue lengths including the average queueing delay.
% We introduce the following definition first.
 Suppose that $\mathbb{Z}_+$ be the set of non-negative integers and $\mathbb{Z}_+^N$ be the $N$ dimensional Cartesian space of non-negative integers.
 %, i.e., $\mathbb{Z}_+^N=\{(x_1,x_2,...,x_N) \mid x_n \in \mathbb{Z_+},n=1,2,...,N \}$. 
We define relation $"\preceq``$ over $\mathbb{Z}_+^N$ as follows. 
 
\begin{defin}
For two vectors $x$ , $\tilde{x} \in \mathbb{Z}_+^N$, we write $\tilde{x} \preceq x$ if one of the following relations holds:
\begin{enumerate}[]
\item \textbf{D1}: $\tilde{x}_n \leq x_n$ for all $n=1,2,...,N$
\item \textbf{D2}: $\tilde{x}$ is obtained by permutation of two distinct elements of $x$, i.e., $\tilde{x}$ and $x$ are different in only two elements $n$ and $m$ such that $\tilde{x}_{n}=x_m$ and $\tilde{x}_{m}=x_n$.
\item \textbf{D3}: $\tilde{x}$ and $x$ are different in only two elements $n$ and $m$ such that $x_n <\tilde{x}_n\leq\tilde{x}_m< x_m$ and the following constraints are satisfied: $\tilde{x}_n=x_n+1$ and $\tilde{x}_m=x_m-1$.
\end{enumerate}
\end{defin}
 In \textbf{D3}, we say that $\tilde{x}$ is more balanced than $x$ and can be obtained by decreasing a larger element of $x$ (between $m$ and $n$) by ``1" and increasing a smaller element (between $m$ and $n$) by ``1". We call such an interchange a \textit{balancing interchange} on vector $x$. Thus, the result of a balancing interchange on a vector $x$ would be a vector $\tilde{x}$ such that $\tilde{x} \preceq x$. Suppose that vector $x \in \mathbb{Z}_+^N$ represents the queue length vector at a given time slot. Then, a balancing interchange is equivalent to taking a packet from a larger queue and adding it to a smaller queue.

We define the partial order $"\preceq_p``$ on $\mathbb{Z}_+^N$ as the transitive closure of relation $"\preceq``$ \cite{lidl}. In other words, $\tilde{x} \preceq_p x$ if and only if $\tilde{x}$ is obtained from $x$ by performing a sequence of 
reductions, permutations of two elements and/or balancing interchanges. When $x$ and $\tilde{x}$ are two queue length vectors, we write $\tilde{x} \preceq_p x$ if and only if queue length vector $\tilde{x}$ is obtained from $x$ by applying a series of packet removal, two queues permutations and balancing interchanges.

We define $\cal F$ as the class of real-valued functions on $\mathbb{Z}_+^N$ that are monotone and non-decreasing with respect to the partial order $"\preceq_p``$, i.e., 
\begin{eqnarray} \label{functions}
f\in \mathcal{F}~~\Longleftrightarrow~~\tilde{x}\preceq_p x \Rightarrow f(\tilde{x}) \leq f(x).
\end{eqnarray}
We can easily check that function $f(x)= \sum_{n=1}^N x_n $ belongs to $\cal F$. This function captures the total queue occupancy of the system.

Let $X'(t)=(X'_1(t),X'_2(t),...,X'_N(t))$ denote the queue length vector at time slot $t$ exactly after serving the queues according to a server assignment policy $\pi$ and before adding the new arrivals of time slot $t$, i.e.,
\begin{eqnarray}
X'_n(t)= \left (X_n(t-1)-\displaystyle\sum_{k=1}^{K}C_{n,k}(t)I_{n,k}^{(\pi)}(t) \right)^+.
\end{eqnarray}
Given $x'(t)$ as a sample value of random variable $X'(t)$, we define a \textit{balancing server reallocation} at time slot $t$ as follows: 

\begin{defin}\label{balancing}
A balancing server reallocation on vector $x'(t)$ is a matching that results in vector $\tilde{x}'(t)$ such that one of the following conditions is satisfied.
\begin{enumerate}[]
\item ($\textbf{C1}$): $\tilde{x}'_n(t) \leq x'_n(t)$ for all $n=1,2,...,N$ and there exists $m \in \{1,2,...,N\}$ such that $\tilde{x}'_m(t) < x'_m(t)$.
\item($\textbf{C2}$): $\tilde{x}'(t)$ and $x'(t)$ are different in only two elements $n$ and $m$ such that $x'_n(t) <\tilde{x}'_n(t)\leq \tilde{x}'_m(t)< x'_m(t)$ and the following constraints are satisfied: $\tilde{x}'_n(t)=x'_n(t)+1$ and $\tilde{x}'_m(t)=x'_m(t)-1$.
\end{enumerate}
\end{defin}
Figures \ref{examples:1} and \ref{examples:2} show two examples of balancing server reallocations in two sample graphs. In these figures, the original allocations are specified by solid lines while the balancing reallocations are specified by dashed lines.

\begin{figure}
  \centering
  \subfloat[Satisfying condition \textbf{C1}]{\label{examples:1}\includegraphics[width=0.22\textwidth]{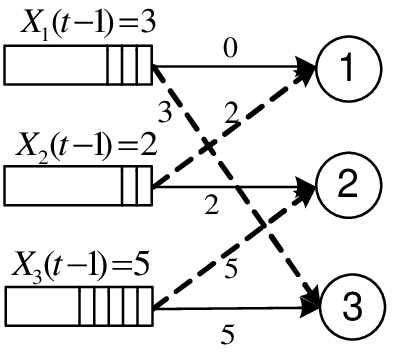}}              
  \subfloat[Satisfying condition \textbf{C2}]{\label{examples:2}\includegraphics[width=0.24\textwidth]{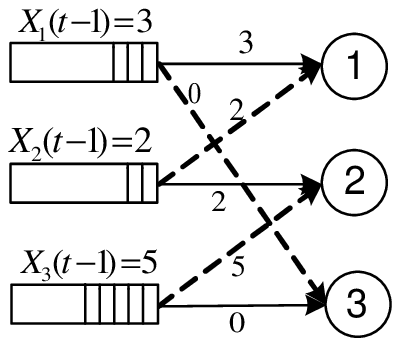}}
  \caption{Examples of balancing server reallocations}
  \label{examples}
\end{figure}
Consider an arbitrary server assignment policy $\pi$ with the allocation variables $\{I_{n,k}^{(\pi)}(t)\}_{t=1}^{\infty}$ for all $k\in \cal K$ and $n \in \cal N$. 
%Thus, we have
%\begin{eqnarray}
%x'_n(t)= \left (x(t-1)-\displaystyle\sum_{k=1}^{K}C_{n,k}(t)I_{n,k}^{(\pi)}(t) \right)^+
%\end{eqnarray}
We introduce Matching Weight ($\mathsf{MW}$) index associated to a server allocation policy $\pi$ at time slot $t$ by 
\begin{eqnarray}
\mathsf{MW}_{\pi}(t)=\displaystyle\sum_{n=1}^{N} x_n(t-1) \displaystyle\sum_{k=1}^{K}c_{n,k}(t) I_{n,k}^{(\pi)}(t) 
\end{eqnarray}
Note that $\mathsf{MW}$ index is exactly the objective of the optimization problem (\ref{matching}). 
According to Definition \ref{balancing} and definition of $\mathsf{MW}$ index, we can prove the following Lemma.
\begin{lem}\label{l1}
For a given policy $\pi$ employing matching $M^{(\pi)}(t)$ at time slot $t$, by applying a balancing server reallocation at time slot $t$ (if there exists any) we will have a new policy $\tilde{\pi}$ differing from $\pi$ only at time slot $t$ such that $\mathsf{MW}_{\pi}(t) < \mathsf{MW}_{\tilde{\pi}}(t)$.
\end{lem}
The proof is omitted here due to space limitations. The detailed proof of the lemma is given in \cite{TR1}.
Based on Lemma \ref{l1}, we can state the following corollary.
\cor \label{cor1}
For a given policy $\pi$ at time slot $t$, if $\mathsf{MW}_{\pi}(t)$ is maximized, i.e., policy $\pi$ employs a maximum weighted matching at time slot $t$, then there exists no balancing server reallocation at that time slot.

Note that Lemma \ref{l1} just states that any balancing reallocation increases the matching weight index. However, it does not imply the existence of a balancing server reallocation when $\mathsf{MW}_{\pi}(t)$ is not maximized. In the following, we will prove the reverse of Lemma \ref{l1}.
\begin{lem}\label{l2}
For a given policy $\pi$ at time slot $t$, if $\mathsf{MW}_{\pi}(t)$ is not maximized, i.e., $\mathsf{MW}_{\pi}(t)<\mathsf{MW}_{\text{MWM}}(t)$, then there exists a balancing server reallocation at that time slot.
\end{lem}

The proof is lengthy and is omitted here due to space limitations. For the detailed proof, please refer to \cite{TR1}.

By $\Pi^\text{MWM}$, we denote the set of all policies who employ maximum weighted matching at all time slots. We also define $\Pi_t$ as the set of all policies that employ maximum weighted matching exactly until time slot $t$ (including $t$). We can easily observe that $\Pi_t \subseteq \Pi_{t-1}$ and $ \Pi^\text{MWM}=\bigcap_{t=1}^{\infty}\Pi_t$.
From Lemmas \ref{l1} and \ref{l2} we conclude that given a policy $\pi \in \Pi_{t-1}$ which is using an arbitrary matching at time slot $t$, we can reach to a policy $\pi^{\star} \in \Pi_t$ by applying a sequence of balancing server reallocations. Suppose that $h_t^{\pi}$ represents the number of balancing server reallocations required to convert the employed matching in policy $\pi$ at time slot $t$ to a maximum weighted matching. In this case, we say that the distance of $\pi$ from $\Pi_t$ is $h_t^{\pi}$ balancing server reallocations. 
%For $h_t^{\pi}$ we have $0 \leq h_t^{\pi}\leq H=\max\{N,K\}$.
Note that if the distance of $\pi$ from $\Pi_t$ is $h_t^{\pi}$, after applying the first balancing server reallocation, we get to a policy $\tilde{\pi}$ whose distance from $\Pi_t$ is $h_t^{\pi}-1$ balancing server reallocations. By repeating this procedure we finally get to a policy whose distance to $ \Pi_t $ is zero, i.e., it belongs to $ \Pi_t $. By $\Pi_t^h$ ($0 \leq h \leq h_t^{\pi}$) we denote the set of all server assignment policies in $ \Pi_{t-1}$ whose distance from $\Pi_t$ is at most $h$ balancing sever reallocations. Note that $\Pi_t^0=\Pi_t$. 

Consider any two policies $ \pi$ and $\tilde{\pi}$ such that $f(\tilde{X}) \leq_{st} f(X)$, $f \in \cal F$ where $X=\{X(t)\}_{t=1}^{\infty}$ and $\tilde{X}=\{\tilde{X}(t)\}_{t=1}^{\infty}$ are the queue length processes when policies $\pi$ and $\tilde{\pi}$ are applied respectively. For such a system, we say policy $\tilde{\pi}$ \textit{dominates} $\pi$. Therefore, if $\tilde{\pi}$ dominates $\pi$ we have $ E[f(\tilde{X})] \leq E[f(X)] $. Given $f(x)= \sum_{n=1}^N x_n $, we conclude that the average queue occupancy (or equivalently average queueing delay) of policy $\tilde{\pi}$ is smaller than that of policy $\pi$.  
According to the above discussion, we can prove the following Lemma.

\begin{lem} \label{l3}
For any policy $ \pi \in \Pi_t^h$ and $0 < h \leq h_t^{\pi}$ we can construct a policy $\tilde{\pi} \in \Pi_t^{h-1}$ 
such that $\tilde{\pi}$ dominates $\pi$.
\end{lem}

Here, we just give the outline of the proof. For the detailed proof please refer to \cite{TR1}.
The proof follows by applying dynamic coupling method over random variables $C(t)=(C_{n,k}(t)),\forall n \in \mathcal{N},\forall k \in \mathcal{K}$ and $A(t)=(A_1(t),A_2(t),...,A_N(t))$. In other words, we will show that given an arbitrary sample path $\omega=(x(0),c(1),a(1),x(1),c(2),a(2),x(2),c(3),a(3),x(3)...)$ we can construct policy $\tilde{\pi}$ and a new sample path $\tilde{\omega} = (\tilde{x}(0), \tilde{c}(1), \tilde{a}(1), \tilde{x}(1), \tilde{c}(2), \tilde{a}(2),\tilde{x}(2),\tilde{c}(3), \tilde{a}(3),\tilde{x}(3), ...)$ resulting in a new sequence of random variables $(\tilde{X}(0), \tilde{C}(1), \tilde{A}(1), \tilde{X}(1), \tilde{C}(2), \tilde{A}(2),\tilde{X}(2),\tilde{C}(3),\hspace{-2pt}...)$ with $X(0)=\tilde{X}(0)$ such that $\tilde{x}(t)  \preceq_{p} x(t)$ \textit{for all} $t$. In fact, we construct $\tilde{\omega}$ and $\tilde{\pi}\in \Pi_t^{h-1}$ in such a fashion that \textit{for all} the sample paths and all time slots we have $\tilde{x}(t)  \preceq_{p} x(t)$. The construction of $\tilde{\pi}$ is consisting of two main steps: construction for time slots before and including $t$ and construction for time slots after $t$. The construction before and including $t$ follows by using the matchings of policy $\pi$ for time slots before $t$. For time slot $t$, we apply the balancing server reallocation. The construction after $t$ follows by using mathematical induction. The detailed proof is lengthy and is omitted at this point. We refer the interested readers to \cite{TR1} for more detail. 

Based on Lemma \ref{l3}, we can prove the main result of this paper in the following Theorem.

\begin{theor}
Maximum Weighted Matching policy dominates any server assignment policy. 
\end{theor}

\begin{proof}
Let $\pi_0$ be any arbitrary policy. Then $\pi_0 \in \Pi_0=\Pi_1^{H_1}$ where $H_1={h^{\pi_0}_1}$. By applying Lemma \ref{l3} repeatedly, we can construct a sequence of policies such that each policy dominates the previous one. Thus, we obtain policies that belong to $\Pi_0=\Pi_1^{H_1},\Pi_1^{{H_1}-1},\Pi_1^{{H_1}-2},...,\Pi_1^0=\Pi_1$. The last policy is called $\pi_1$. Note that $\pi_1 \in \Pi_2^{H_2}$ where $H_2=h^{\pi_1}_2$. By recursively continuing such argument we obtain a sequence of policies $\pi_t \in \Pi_t$, $t=1,2,...$ such that $\pi_j $ dominates $\pi_i$ for $j>i$.
Note that this sequence of policies defines a limiting policy $\pi^\ast$ that agrees
with MWM at all time slots. Thus, $\pi^\ast $ is an MWM policy who dominates all the
previous policies, including the starting policy $\pi_0$.
\end{proof}

%%%%%%%%%%%%%%%%%%%%%%%%%%%%%%%%%%%%%%%%%%%%%%%%%%%%%%%%%%%%%%%%%%%%%%%%%%%%%%%%%%%%%%%%%%%%%%%%
%%%%%%%%%%%%%%%%%%%%%%%%%%%%%%%%%%%%%%%%%%%%%%%%%%%%%%%%%%%%%%%%%%%%%%%%%%%%%%%%%%%%%%%%%%%%%%%%
\vspace*{1mm}
\section{Conclusions}\label{conc}
In this paper, we considered the problem of assignment of $K$ identical servers to a set of $N$ parallel queues in a symmetrical time slotted queueing system with random connectivities from the queues to the servers. For such a queueing system, it has been previously shown that MWM is throughput optimal, i.e. has the maximum stability region. Our contribution in this work is the development of a method to prove the optimality of MWM in minimizing, in stochastic ordering sense, a class of cost functions of queue lengths (including total queue occupancy or equivalently average queueing delay). Our method to achieve this goal used stochastic ordering and dynamic coupling techniques.
%
%Using stochastic ordering and dynamic coupling method we proved that for a symmetric system (i.e. with the same arrival and connectivity parameters for all the queues), MWM is also optimal in minimizing,  
\vspace*{1mm}
\IEEEtriggeratref{16}
\bibliographystyle{ieeetran}
%\vspace*{2mm}
\bibliography{Ref}

\end{document}